\documentclass[12pt]{article}

\usepackage{amsthm}
\usepackage{latexsym}
\usepackage{amssymb,amsmath}
\usepackage{epsfig}
\usepackage{amsfonts}
\usepackage{graphicx}
\usepackage{color}
\usepackage{algorithm}
\usepackage{algorithmic}
\usepackage{caption}
\usepackage{url}
\usepackage{cite}

\date{}

\newtheorem{theorem}{Theorem}[section]

\newtheorem{proposition}[theorem]{Proposition}
\newtheorem{corollary}[theorem]{Corollary}

\newtheorem{remark}[theorem]{Remark}

\newcommand{\Cay}{{\rm Cay}}
\newcommand{\diam}{{\rm diam}}
\newcommand{\T}{\mathrm}
\newcommand{\dd}{{\rm d}}
\newcommand{\Z}{\mathbb{Z}}
\newcommand{\fm}{\frak{m}}

\DeclareUnicodeCharacter{200B}{}
\DeclareUnicodeCharacter{041D}{}
\DeclareUnicodeCharacter{0430}{}
\DeclareUnicodeCharacter{0443}{}
\DeclareUnicodeCharacter{0447}{}
\title{Packing chromatic number of unitary Cayley graphs of $\Bbb Z_n$ and algorithmic approaches to it}

\author{Zahra Hamed-Labbafian $^{a,}$\thanks{Email: \texttt{hamedlabbafianzahra@gmail.com}} 
\and  
Mostafa Tavakoli $^{a,}$\thanks{Corresponding author, Email: \texttt{m$\_$tavakoli@um.ac.ir}} 
\and
Mojgan Afkhami $^{b,}$\thanks{Email: \texttt{Afkhami@neyshabur.ac.ir}} 
\and 
Sandi Klav\v{z}ar $^{c,d,e,}$\thanks{Email: \texttt{sandi.klavzar@fmf.uni-lj.si}}
}

\begin{document}

\maketitle

\begin{center}
$^a$ Department of Applied Mathematics,
Faculty of Mathematical Sciences, \\
Ferdowsi University of Mashhad,
Mashhad, 91775 Iran\\
\medskip

$^b$ Department of Mathematics, University of Neyshabur,\\
P.O.Box 91136-899, Neyshabur, Iran\\
\medskip

$^c$ Faculty of Mathematics and Physics, University of Ljubljana, Slovenia\\
\medskip

$^d$ Institute of Mathematics, Physics and Mechanics, Ljubljana, Slovenia\\
\medskip

$^e$ Faculty of Natural Sciences and Mathematics, University of Maribor, Slovenia
\end{center}

\date{}
\maketitle

\begin{abstract}
A packing $k$-coloring of a graph $G$ is a partition of $V(G)$ into $k$ disjoint non-empty classes $V_1, \dots, V_k$, such that if $u,v \in V_i$, $i\in [k]$, $u\ne v$, then the distance between $u$ and $v$ is greater than $i$. The packing chromatic number of $G$ is 
the smallest integer $k$ which admits a packing $k$-coloring of $G$. In this paper, the packing chromatic number of the unitary Cayley graph of $\mathbb{Z}_n$ is computed. Two metaheuristic algorithms for calculating the packing chromatic number are also proposed.
\end{abstract} 

\\

\noindent
{\bf AMS Subj.\ Class.\ (2020)}:  05C15, 05C85

\section{Introduction}
 
Packing coloring is a variant of coloring with applications in frequency assignment, resource allocation, and wireless network design. In this variant, vertices of a graph are assigned colors such that vertices sharing the same color are separated by a distance determined by the color itself. More precisely, a {\it packing $k$-coloring} of a graph $G = (V(G), E(G))$ is a partition of $V(G)$ into $k$ disjoint non-empty color classes $V_1, \dots, V_k$, such that $V_i$, $i\in [k]$, is an {\it $i$-packing}, that is, for each two distinct vertices $u,v \in V_i$ we have $\dd_G(u,v) \ge i+1$. The smallest integer $k$ which admits a packing $k$-coloring of  $G$ is  the {\it packing chromatic number} $\chi_{\rho}(G)$ of $G$. 

The packing chromatic number was initially explored under the name broadcast chromatic number by Goddard et al.\ in~\cite{Goddard}. The terminology and notation used today was proposed in~\cite{Bresar}. This coloring concept has already been extensively and deeply researched, and the review article~\cite{Bresar2} published in 2020 contains 68 references. Research continued with recent studies of the packing chromatic number of iterated {M}ycielskians~\cite{bidine-2023} and hypercubes~\cite{gregor-2024}, and with investigations of variants such as distance dominator packing coloring~\cite{ferme-2021}, partial and quasi-packing packing coloring~\cite{grochowski-2025}, and Grundy packing coloring~\cite{Greedy}. Several recent papers deal also with criticality concepts, cf.~\cite{ferme-2022, klavzar-2023}. The greatest emphasis in recent times, however, has been on $S$-packing colorings, especially on subcubic graphs, see~\cite{bresar-2025, kostochka-2021, liu-2020, mortada-2024, mortada-2025, yang-2023}. 

It should be stressed that it is intrinsically difficult to determine the packing chromatic number. It was proved already in the seminal paper~\cite{Goddard} that the decision problem whether an input graph admits a packing $k$-coloring is NP-complete for $k=4$, even when restricted to planar graphs. This finding was followed up by Fiala and Golovach~\cite{fiala-2010} with the breakthrough result asserting that this decision problem is is NP-complete for trees. Furthermore, it was later proven that the decision problem remains NP-complete when restricted to chordal graphs with diameter at least $3$~\cite{kim-2018}.

For these reasons, it is desirable to find different heuristic and/or approximation algorithmic approaches to the packing chromatic number. We do this in Section~\ref{sec:algorithms}  by considering the  local search and the genetic algorithm for computing the packing chromatic number. In Section~\ref{sec:experiment} we then report experimental results on these two approaches and make their comparison with the recently proposed greedy approach~\cite{Greedy}. Before turning our attention to algorithms, we determine in Section~\ref{sec:unitary-Cayley} the packing chromatic number of unitary Cayley graphs of $\Bbb Z_n$. 

In the reminder of this section we recall some definitions and notations, for undefined terms we refer to~\cite{graph}. Let $G = (V(G), E(G))$ be a graph. We use the notation $x\sim_G y$ to denote that $xy\in E(G)$. A subset $S$ of $V(G)$ is \textit{independent} if no two vertices of $S$ are adjacent. The cardinality of a largest independent set is the \textit{independence number} $\alpha(G)$ of $G$. The \textit{distance} $\T{d}_G(a,b)$ between vertices $a$ and $b$ of a connected graph $G$ is the length of a shortest $a,b$-path. The diameter $\diam(G)$ of $G$ is the length of a longest shortest path in $G$. By $K_{n_1,\dots,n_m}$ we denote the complete $m$-partite graphs with parts of size $n_i$, $i\in [m]$. The {\it direct product} $G_1\times G_2$ of graphs $G_1$ and $G_2$ is the graph with vertex set  $V(G_1) \times V(G_2)$ and $(u_1,v_1) \sim_{G_1\times G_2} (u_2,v_2)$ if $u_1 \sim_{G_1} u_2$ and $v_1 \sim_{G_1} v_2$.  

\section{Packing chromatic number of unitary Cayley graphs of $\Bbb Z_n$}
\label{sec:unitary-Cayley}

In this section we determine the packing chromatic number of the unitary Cayley graph of $\Z_n$, the ring of integers modulo $n$.

Let $R$ be a finite commutative ring with nonzero identity, and let $R^\times$ denote the set of all unit elements of $R$. The {\it unitary Cayley graph} $G_R=\text{Cay}(R,R^{\times})$ of $R$ is a graph with the vertex set $R$ and two vertices $x$ and $y$ are adjacent if $x-y \in R^{\times}$. We refer to~\cite{burcroff-2023, dolzan-2024, ratt-2024} for some recent investigations of unitary Cayley graphs, see also~\cite{Akhtar, EUZn, U1}.

If $R$ is a finite commutative ring, then by~\cite[p.~752]{AAlg} we can write $R\cong R_1 \times\dots \times R_t$, where  $R_i$, $i\in [t]$, is a finite local ring with maximal ideal $\fm_i$.  This decomposition is unique up to permutation of factors. We denote  the (finite) residue field $\frac{R_i}{\fm_i}$ by $K_i$ and  $f_i = |K_i|=\frac{|R_i|}{|\fm_i|}$. We also assume (after appropriate permutation of factors) that $f_1\leqslant \dots \leqslant  f_t$.

The following proposition is a basic consequence of the definition of the unitary Cayley graphs, cf.~\cite[Proposition 2.2]{Akhtar}. 

\begin{proposition} 
\label{basic}
If $R$ is a finite commutative ring, then the following statements hold.
\begin{itemize} 
\item[(a)] The graph $G_R$ is a $|R^{\times}|$-regular graph.
\item[(b)]  If $R$ is a local ring with maximal ideal $\fm$,  then $G_R$ is a complete multipartite graph whose partite sets are the cosets of $\fm$ in $R$. In particular, $G_R$ is a complete graph if and only if $R$ is a field.
\item[(c)] If $R \cong R_1 \times \cdots \times R_t$ is a product of local rings, then $G_R \cong \times_{i=1}^{t} G_{R_i}$. Hence, $G_R$ is the direct product of complete multipartite graphs.
\end{itemize}
\end{proposition}

\begin{proposition} {\rm \cite[Theorem 3.1]{Akhtar}}
\label{Akhtar}
If $R\cong R_1 \times\dots \times R_t$  is a finite, commutative ring, then
\begin{eqnarray*}
	{\rm diam}(G_{R})=\left\lbrace \begin{array}{ll}
		1; & t=1 \ \T{and} \ R \ \T{is \ a \ field},\\
	2; & t=1 \ \T{and} \ R \ \T{is  \ not \ a \ field},\\
	2; & t \geqslant 2, f_1 \geqslant 3,\\
	3; & t \geqslant 2, f_1=2, f_2 \geqslant 3,\\
		\infty; & t \geqslant 2, f_1=f_2=2.\\
	\end{array}\right.
\end{eqnarray*}
\end{proposition}

For the rest of this section, some preparation is needed. Let $n=p_1^{r_1}\dots p_t^{r_t}$ be the prime factorization of $n$, where $p_i$s are primes with $p_1 < \dots < p_t$. Then $\Z_n \cong \Z_{p_1^{r_1}} \times \dots \times \Z_{p_t^{r_t}}$. Also  $\Z_{p_i^{r_i}}$ is a local ring with the maximal ideal $\fm_i=\{rp_i \mid r\in \Z_{p_i^{r_i}}\}$ with $|\fm_i|=p_i^{r_i-1}$ and the number of cosets of $\fm_i$ in $\Z_{p_i^{r_i}}$ is equal to $p_i$, for each $i \in [t]$.
 
\begin{remark}\label{struc}
By Proposition~\ref{basic} and the above preparation, $G_{\Z_{p_i^{r_i}}}$, $i\in [t]$, is isomorphic to the complete $p_i$-partite graph $K_{p_i^{r_i-1}, \dots ,p_i^{r_i-1 }}$. Also, since $\Z_n \cong \Z_{p_1^{r_1}} \times \dots \times \Z_{p_t^{r_t}}$, the third part of Proposition~\ref{basic}, yields
 \[G_{\Z_n}\cong \times_{i=1}^{t} K\underbrace{_{p_i^{r_i-1}, \dots ,p_i^{r_i-1 }}}_{p_i}.\]
 \end{remark}

Using Proposition~\ref{Akhtar}, the following result can be deduced.
 
\begin{corollary}\label{diam}
If $p_1^{r_1}\dots p_t^{r_t}$ is the prime factorization of $n$, where $p_1 < \dots < p_t$, then $G_{\Z_n}$ is connected and we have
\begin{eqnarray*}
	{\rm diam}(G_{\Z_n})=\left\lbrace \begin{array}{ll}
		1; & t=1, r_1 =1, \\
	2; & t=1, r_1 >1, \\
	2; & t \geqslant 2, p_1 \geqslant 3, \\
	3; & t \geqslant 2, p_1=2.\\
			\end{array}\right.
\end{eqnarray*}
\end{corollary}

Now all is ready to determine the packing chromatic number of $G_{\Z_n}$. The case when $n$ is a prime is trivial since in that case $G_{\Z_n}$ is a complete graph. By Corollary~\ref{diam} the remaining cases are of diameter two or three. For the first case we recall the following result from the seminal paper~\cite{Goddard}, see also the survey~\cite[Proposition 2.5]{Bresar2}.

\begin{proposition}
\label{diam2}
If $G$ is a graph, then
\[\chi_\rho(G) \le n(G)-\alpha(G)+1,\]
with equality if 
${\rm diam}(G)=2$.
\end{proposition}

Assume that $n=p^r$, where $p$ is a prime number and $r>1$.
In this situation, by Corollary \ref{diam}, we have ${\rm diam}(G_{\Z_{p^r}})=2$, and so by 
Proposition \ref{diam2}, we have 
$\chi_\rho(G_{\Z_{p^r}})=p^r-\alpha(G_{\Z_{p^r}})+1$.
By Remark \ref{struc}, $G_{\Z_{p^r}}$ is isomorphic to the $p$-partite $K_{p^{r-1}, \dots ,p^{r-1 }}$, for which we have  $\alpha(G_{\Z_{p^r}})=p^{r-1}$.  Therefore, if $n=p^r$, $r>1$, then 
$$\chi_\rho(G_{\Z_n})=p^r-p^{r-1}+1.$$
The second subcase when the diameter of the unitary Cayley graph is two in covered with the next result. 

\begin{theorem}\label{33}
If $n=p_1^{r_1} \dots p_t^{r_t}$, where $t>1$ and $3 \leqslant p_1 < \dots < p_t$, then
$$\chi_\rho(G_{\Z_n})=p_1^{r_1-1} p_2^{r_2} \dots p_t^{r_t}(p_1-1)+1.$$
 \end{theorem}
 
\begin{proof}
By Corollary~\ref{diam}, ${\rm diam}(G_{\Z_{n}})=2$. Hence by Proposition~\ref{diam2} it suffices to determine $\alpha(G_{\Z_{n}})$. By Remark~\ref{struc}, $G_{\Z_{n}}$ is isomorphic to the direct product of $t$ $p_i$-partite graphs $K_{p_i^{r_i-1}, \dots ,p_i^{r_i-1 }}$, $i \in [t]$. Also, $G_{\Z_{p_i}}$, $i \in [t]$, is isomorphic to the $p_i$-partite graph $K_{p_i^{r_i-1}, \dots ,p_i^{r_i-1 }}$. Let $V^{p_i}_1$, $\dots$, $V^{p_i}_{p_i}$ be the $p_i$ parts of $G_{\Z_{p_i}}$, and $V ^{p_i}_j=\{a^{p_i}_{j,1}, \dots, a^{p_i}_{j,p_i^{r_i-1}}\}$, $j \in [p_i]$. Clearly $|V ^{p_i}_j|=|V ^{p_i}_{j'}|=p_i^{r_i-1}$, for each $j ,j' \in [p_i]$. For $j\in [p_i]$ set 
\[A^{p_i}_j=V(G_{\Z_{p_1}}) \times \dots \times V(G_{\Z_{p_{i-1}}}) \times V^{p_i}_j \times V(G_{\Z_{p_{i+1}}}) \times \dots \times V(G_{\Z_{p_t}}).\]
The set $A^{p_i}_j$, $j \in [p_i]$, is an independent set of order $p_1^{r_1}\dots p_{i-1}^{r_{i-1}} p_i^{r_i-1}  p_{i+1}^{r_{i+1}}\dots p_t^{r_t}$, and 
\[V(G_{\Z_{n}})=\bigcup_{i=1}^{t}\bigcup_{j=1}^{p_i} A^{p_i}_j.\]
Since $ p_1 < \dots < p_t$ and $| \bigcup_{j=1}^{p_1} A^{p_1}_j|=\dots=|\bigcup_{j=1}^{p_t} A^{p_t}_j|$, we have $|A^{p_1}_1|> \dots > |A^{p_t}_1|$.

Let $S$ be an arbitrary independent set of $G_{\Z_n}$. Since there are at most $p_1^{r_1-1}$ possibilities for the first component of the elements of $S$,  we must have $|S| \leqslant p_1^{r_1-1}p_2^{r_2}\dots p_{t}^{r_t}$, which means that $|S| \leqslant |A^{p_1}_1|$. 
Therefore, $A^{p_1}_1$ is an independent set of greatest order, which means that $\alpha(G_{\Z_n})=|A^{p_1}_1|=p_1^{r_1-1}p_2^{r_2}\dots p_{t}^{r_t}$. Proposition~\ref{diam2} completes the argument. 
\end{proof}

In view of Corollary~\ref{diam}, the only remaining case to be considered is the following. 

\begin{theorem}\label{4}
If $n=p_1^{r_1} \dots p_t^{r_t}$, where $t>1$, $p_1=2$, and $ p_1 < \dots < p_t$, then $$\chi_\rho(G_{\Z_n})=p_1^{r_1-1} p_2^{r_2} \dots p_t^{r_t}(p_1-1).$$
 \end{theorem}
 
\begin{proof}
By the proof of Theorem \ref{33}, the greatest size of a 1-packing of $G_{\Z_n}$ is $p_1^{r_1-1} p_2^{r_2} \dots p_t^{r_t}$. By Corollary~\ref{diam}, each two distinct vertices of $G_{\Z_n}$ are at distance at most three. Also, by Remark~\ref{struc}, we have $G_{\Z_{2^{r_1}}}\cong  K_{2^{r_1-1}, 2^{r_1-1 }}$. Let $V_1$ and $V_2$ be the two parts of $G_{\Z_{2^{r_1}}}$. If $X=(x_1,\dots,x_t)$ and $Y=(y_1,\dots,y_t)$ are nonadjacent vertices of $G_{\Z_n}$, then $\dd_{G_{\Z_n}}(X,Y)=3$ if and only if $x_1 \in V_1$ and $y_1 \in V_2$, or vice versa. Now let $A$ be a 2-packing in $G_{\Z_n}$. If $|A|>2$, then there are three distinct vertices  $U_1=(x_1,\dots,x_t)$,  $U_2=(y_1,\dots,y_t)$, and  $U_3=(z_1,\dots,z_t)$ in $A$ such that $\dd_{G_{\Z_n}}(U_i,U_j)=3$, for each $1\leqslant i\neq j \leqslant 3$. Without loss of generality, assume that $x_1 \in V_1$ and $y_1 \in V_2$. If $z_1 \in V_1$, then we have $\dd_{G_{\Z_n}}(U_1,U_3)=2$, which is impossible. If $z_1 \in V_2$, then we have $\dd_{G_{\Z_n}}(U_2,U_3)=2$, which is again impossible. So any 2-packing in $G_{\Z_n}$ has at most two vertices and for $i >2$, each $i$-packing of $G_{\Z_n}$ has at most one vertex. We can conclude that  $\chi_\rho(G_{\Z_n})=p_1^{r_1-1} p_2^{r_2} \dots p_t^{r_t}(p_1-1)$.
\end{proof}

\section{Two metaheuristic algorithms}
\label{sec:algorithms}

Recall from the introduction that there is a strong case for utilizing meta-heuristic algorithms for the approximation  of the packing chromatic number. In this section we employ the local search and the genetic algorithm for solving this problem. We recall that the genetic algorithm is designed as a population-centric technique, whereas the local search algorithm focuses on point-based optimization.
 
\subsection{Local search algorithm}
\label{sec:LA}

We now present a local search (LS) algorithm for finding a packing $k$-coloring for a given graph $G$. 

The main algorithm's objective is to achieve a packing $k$-coloring by iteratively improving an initial solution. $Maxit$ in Algorithm \ref{alg1} denotes the maximum number of iterations chosen as a stopping criterion. In addition, for given $k$ the notation $c_{\T{fitness}}$ is an evaluation function for a coloring $c$ which is defined as $c_{\T{fitness}}=\frac{1}{1+I_c}$ where $I_c=\sum_{i=1}^k|\{\{u,v\}\subset V(G):\ u\neq v, c(u)=c(v)=i, \dd_G(u,v)\leq i\}|$. This process is performed by the sub-algorithm named the fitness calculation (FC) algorithm (FC for short) for each coloring. Algorithm~\ref{alg1} continues by generating neighboring solutions. If for a specific $k$, a neighbor solution $c^i$ is found such that $c^i_{\T{fitness}}=1$, we have a packing $k$-coloring. Otherwise,  another value of $k$ can be checked. Moreover, $N_c$ denotes a set of neighboring configurations, each of then differ from the current coloring $c$ in only one vertex's color.

\subsection{Detailed explanation of the LS algorithm}

\noindent
{\bf Initialization} \\
We start by generating a random coloring $c$ of $V(G)$ with $k$ colors. Then we apply the FC sub-algorithm to compute the fitness number $c_{\T{fitness}}$.
 If $c_{\T{fitness}}=1$, it’s immediately returned as the best solution.

\medskip\noindent
{\bf Local search loop}\\
If $c_{\T{fitness}}<1$, then the algorithm enters a local search loop. In each iteration, a set of neighboring configurations $N_c$ is generated.
For each neighbor configuration of $N_c$, the algorithm computes the fitness function and updates the current solution if a neighboring solution with grater fitness value is found.

\medskip\noindent
{\bf Updating best solution}\\
If the fitness value of the best neighboring solution is greater than $c_{\T{fitness}}$, then we update $S_{\T{best}}$; In other word,
If $f_{\T{N}}> f_{\T{best}}$, the algorithm sets $S_{\T{best}}=N_{\T{best}}$ and $f_{\T{best}}=f_{\T{N}}$. 
We recall that if $f_{\T{best}}=1$, then the algorithm returns $S_{\T{best}}$ as the best solution. Otherwise, the algorithm increments the iteration count.

\medskip\noindent
{\bf Termination}\\
The process continues until either a coloring $c$ with $c_{\T{fitness}}=1$ is found, the maximum number of iterations (Maxit) is reached, or no further improvement can be made.
Finally, the algorithm returns $S_{\T{best}}$ as the best solution.

\begin{algorithm}[ht!]
\caption{LS algorithm to find a packing $k$-coloring}\label{alg1}
 \begin{algorithmic}[1]
\REQUIRE Graph, $G$, the number of colors $k$, the number of iterations $Maxit$.
\ENSURE A packing $k$-coloring of $G$.\label{3}
\STATE Generate an arbitrary vertex coloring, $c$, for $G$.
\STATE Call FC algorithm for $c$ ($c_{\T{fitness}}\leftarrow FC(G,k,c)$).
\STATE Set $S_{\T{best}}=c$ and $f_{\T{best}}=c_{\T{fitness}}$.
\IF {$f_{\T{best}}=1$}  
\STATE Return $S_{\T{best}}$\label{4}. 
\ELSE
\STATE Let $count \leftarrow 1$.  //count is the iteration counter.
\WHILE {$count<=Maxit$}\label{7}
\STATE Generate neighborhood set $N_c$ of $c$.
\STATE Set $N_{\T{best}}=\emptyset$ and $f_{\T{N}}=0$.
\WHILE {$N_c \ne \varnothing$}
\STATE Select an element, $c^i$ of $N_c$ and call FC for it ($c^i_{\T{fitness}}\leftarrow FC(G,k,c^i)$).
\IF {$c^i_{\T{fitness}}>f_{\T{N}}$}
\STATE $N_{\T{best}}\leftarrow c^i$ and $f_{\T{N}}\leftarrow c^i_{\T{fitness}}$.
\ENDIF
\STATE $N_c\leftarrow N_c-\{c^i\}$.
\ENDWHILE
\IF {$f_{\T{N}}>f_{\T{best}}$}
\STATE Set $S_{\T{best}}=N_{\T{best}}$ and $f_{\T{best}}=f_{\T{N}}$.
\IF {$f_{\T{best}}=1$}
\STATE Return $S_{\T{best}}$ and break.
\ENDIF
\ENDIF
\STATE Set $c\leftarrow N_{\T{best}}$, $count\leftarrow count+1$.
\ENDWHILE
\ENDIF
\STATE Return $S_{\T{best}}$.
\end{algorithmic}
\end{algorithm}

\begin{algorithm}[ht!]
\caption{FC algorithm}\label{alg2}
 \begin{algorithmic}[1]
\REQUIRE graph $G$, number of colors $k$, coloring $c$.
\ENSURE The fitness value $c_{\T{fitness}}$.
\STATE Set $I_c=0$.
\FOR {$i=1$ to $k$}
\FOR {each $u,v\in V(G)$ such that $c(u)=c(v)=i$}
\IF {$\dd_G(u,v)\leq i$}
\STATE {Set $I_c\leftarrow I_c+1$.}
\ENDIF
\ENDFOR
\ENDFOR
\STATE Set $c_{\T{fitness}}=\frac{1}{1+I_c}$.
\STATE Return $c_{\T{fitness}}$.
\end{algorithmic}
\end{algorithm}

\newpage
\subsection{Genetic algorithm}

Here we apply the genetic algorithm for solving the packing coloring problem. 
This algorithm begins by generating an initial population of size $n_p$, and calculating its fitness values via algorithm FC. If a solution achieves the fitness value $1$, then it is immediately returned as the optimal solution. Otherwise, the algorithm proceeds to the main iterative phase.

At the beginning of each iteration, offspring and mutated populations are generated through crossover and mutation.
These populations are merged with the original, and less optimal solutions are removed to create a refined population of size $n_p$. 
This process is repeated until a maximum number of iterations is reached.
Finally, the algorithm selects the solution with the highest fitness value. If the fitness equals 1, it outputs a graph coloring. Otherwise,
the algorithm reports that it was unable to color the graph using $k$ colors.
Below, we describe the combination operations employed by our algorithm to generate a new population.

\medskip\noindent
{\bf Crossover 1} \\
In this process of generating a new offspring, two parents are randomly selected from the existing population. The resulting offspring is initially created as an exact copy of the second parent's genome. Subsequently, a specific modification mechanism is applied: any colors present in the first parent's genetic structure but absent in the second parent are randomly substituted into selected vertices of the offspring's structure.

\medskip\noindent
{\bf Crossover 2}\\
The process of generating new offspring begins with the random selection of two parents from the population. The initial offspring structure is formed through complete replication of the first parent's genome. Subsequently, in a targeted modification process, two vertices are randomly selected from the offspring's structure and their corresponding colors are transferred from the second parent. This genetic transfer mechanism ensures a balanced combination of characteristics from both parents.

\medskip\noindent
{\bf Mutation}\\
The process begins with random selection of a parent from the existing population. The mutated's initial structure is established through complete replication of the parent's genome. Subsequently, in a targeted modification process, two vertices are randomly selected from the mutated's structure, and their colors are interchanged. This mutation mechanism enhances genetic diversity within the population and facilitates the exploration of new regions in the search space.

\medskip
GA is presented as Algorithm 3.

\begin{algorithm}[ht!]
\caption{GA for the packing coloring problem}\label{alg3}
\begin{algorithmic}[1]
\REQUIRE $G$, $k$, maximum number of iterations $Maxit$, initial population's size $n_p$.
\ENSURE a $k$-packing coloring for $G$.
\STATE Create an initial populations with size $n_p$, and evaluate them using FC algorithm.
\STATE $count=1$.
\WHILE {$count\leq Maxit$}
\STATE Select parents randomly.
\STATE Generate and evaluate offspring using the crossover operators.
\STATE Generate and evaluate mutated populations using the mutation operator.
\STATE Merge the initial population, offspring, and mutated populations, and sort them based on the fitness function (new population).
\STATE Truncate new population and generate a new population with size $n_p$.
\STATE  $count=count+1$.
\ENDWHILE
\STATE Return the individual with the best fitness value as the best solution.
\end{algorithmic}
\end{algorithm}

\section{Experimental results}
\label{sec:experiment}

The section includes two parts: the initial part reviews the efficiency of GA, while the latter part delivers a comparative analysis of the three algorithms: LS, GA, and Greedy.
In what follows, $pc_{GA}$ and $pc_{Greedy}$ refer to the approximations of the packing chromatic number computed by the genetic algorithm and the greedy algorithm, respectively.
\subsection{GA efficiency}

Here we focus on assessing the effectiveness of the earlier described GA. We first present in Table~\ref{Tab1} a report on the algorithm’s output for graphs with known exact chromatic numbers as provided in~\cite{Bresar2}.

\begin{table}[ht!]
\caption{$\chi_\rho$ of some graphs computed by GA} \label{Tab1}
\begin{center}
\begin{tabular}{|c||c|c|c|}
		\hline
Graph name&$\chi_{\rho}$&$pc_{GA}$&GA CPU time (s)\\ \hline \hline 
$C_{15}$&4&4&0.458  \\ \hline
$C_{20}$&3&3&3.442 \\ \hline
$P_{20}$&3&3&5.721\\ \hline
$S_{10}$&2&2&0.120\\ \hline
$K_{3,5,7}$&9&9&0.372\\ \hline
\end{tabular}
\end{center}
\end{table}

GA has been next tested on some examples from Section~\ref{sec:unitary-Cayley}, see Table~\ref{Tab2}.

\begin{table}[h!]
\caption{$\chi_\rho$ of some unitary Cayley graphs computed by GA} \label{Tab2}
\begin{center}
\begin{tabular}{|c||c|c|c|}
		\hline
Graph name&$\chi_{\rho}$&$pc_{GA}$&GA CPU time (s)\\ \hline \hline 
$G_{\Z_{5}}$&5&5&0.002  \\ \hline
$G_{\Z_{16}}$&9&9&0.299 \\ \hline
$G_{\Z_{21}}$&15&15&2.103\\ \hline
$G_{\Z_{27}}$&19&19&1.893\\ \hline
$G_{\Z_{45}}$&31&31&22.196\\ \hline
\end{tabular}
\end{center}
\end{table}

\subsection{Comparison of the three algorithms}

Article \cite{Greedy} introduces a greedy algorithm designed to assist in determining the chromatic number of graphs. We compared the results obtained from both the greedy and genetic algorithms. Through an analysis of various graph samples, particularly a subset of Cayley graphs, we found that for some graphs, all three methods found packing colorings of the same cardinality. Some of these examples are grouped together in Table~\ref{Tab3}.

\begin{table}[ht!]
\caption{Graphs for which all three algorithms return the same value} \label{Tab3}
\begin{center}
\begin{footnotesize}
\begin{tabular}{|c||c|c|c|c|}
		\hline
Graph &value&\begin{tabular}{c}
Greedy CPU \\
 time (s)\end{tabular}&
\begin{tabular}{c}
GA CPU\\
 time (s)\end{tabular}
&\begin{tabular}{c}
LS CPU\\
 time (s)\end{tabular}\\ \hline \hline 
$\Cay(\Z_8,\{1,3,5,7\})$&5&0.001&0.039&0.004  \\ \hline
$\Cay(\Z_9,\{1,3,5,7\})$&7&0.002&0.029&0.007 \\ \hline
$\Cay(\Z_{12},\{1,3,9,11\})$&7&0.002&0.769&0.023\\ \hline
$\Cay(\Z_{8},\{1,2,3,5,6,7\})$&7&0.001&0.022&0.002\\ \hline
$\Cay(\Z_{9},\{1,2,3,6,7,8\})$&8&0.002&0.020&0.006\\ \hline
$\Cay(\Z_{12},\{1,2,3,9,10,11\})$&10&0.002&0.054&0.017\\ \hline
\end{tabular}
\end{footnotesize}
\end{center}
\end{table}

Although LS is faster than GA, it might get stuck in a local maximum and miss a better solution. Hence we  moved on to compare the genetic and the greedy algorithm on some larger graphs. In Table~\ref{Tab4} computational results are collected for the generalized Petersen graph $G(12,2)$, the graph BN16 (shown in \ref{fig:BN16}), the truncated icosahedral graph TI and the fullerene graphs $C_{48}$ and $C_{70}$.

\begin{figure}[ht!]
    \centering
    \includegraphics[width=0.4\linewidth]{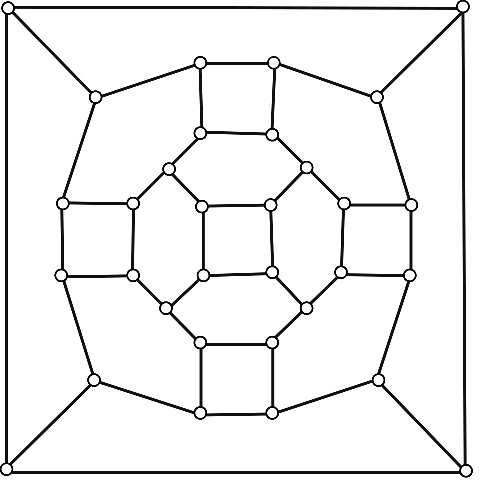}
    \caption{Graph BN16}
    \label{fig:BN16}
\end{figure}

\begin{table}[ht!]
\caption{Computational results for some additional graphs} \label{Tab4}
\begin{center}
\begin{small}
\begin{tabular}{|c||c|c|c|c|c|}
		\hline
Graph $G$ & $n(G)$ &$pc_{Greedy}$&\begin{tabular}{c}
Greedy CPU\\
 time (s)\end{tabular}&$pc_{GA}$&\begin{tabular}{c}
GA CPU\\
 time (s)\end{tabular}\\ \hline \hline 
$G(12,2)$ &24& \begin{tabular}{c}
13,11,13,\\
12,13,12
\end{tabular}&0.007&10&60.343  \\ \hline
BN16 &32&
\begin{tabular}{c}
14,14,13,\\
13,11,12
\end{tabular}&0.011&9&89.511 \\ \hline
$C_{48} $&48&
\begin{tabular}{c}
17,16,17,\\
16,19,18\end{tabular}&0.028&13&443.667\\ \hline
TI &60&\begin{tabular}{c}21,21,21,\\
17,19,23\end{tabular}&0.055&16&549.806\\ \hline
$C_{70}$ &70&\begin{tabular}{c}22,22,26,\\
21,24,25\end{tabular}&0.081&19&1237.402\\ \hline
\end{tabular}
\end{small}
\end{center}
\end{table}

A comparative analysis of the greedy and genetic algorithm reveals that as graph complexity and size increase, the performance gap between these methods becomes more pronounced. For instance, for the fullerene $C_{48}$, the greedy algorithm yielded inconsistent results across multiple iterations (e.g., 16, 17, 16, 17, 19, 18), highlighting its instability and inefficiency for complex graphs. On the other hand, GA found a packing coloring with 13 colors, but using for it 443.667 seconds. Nevertheless, GA demonstrated  its ability to explore a broader solution space and identify a better, most likely optimal result. 

These insights offer guidance for selecting graph coloring strategies in future research, especially for large and complex graphs. If an upper bound for the packing chromatic number can be established using some of the results from ~\cite{Bresar2} or elsewhere, it can serve as input for the genetic algorithm; otherwise, the greedy algorithm can provide an initial estimate, which can then be refined by the genetic algorithm.

\section*{Acknowledgments}

Sandi Klav\v zar were supported by the Slovenian Research and Innovation Agency (ARIS) under the grants P1-0297, N1-0355, and N1-0285. The research of the second author, Mostafa Tavakoli, was supported in part by the Ferdowsi University of Mashhad.

\section*{Declaration of interests}

The authors declare that they have no known competing financial interests or personal relationships that could have appeared to influence the work reported in this paper.

\section*{Data availability}
Our manuscript has no associated data.

 \end{document}